\documentclass[12pt, reqno]{amsart}

\usepackage[paperheight=297mm, paperwidth=210mm, left=30mm, right=20mm]{geometry}
\usepackage{amssymb,latexsym,amsmath,amsfonts,amsthm,amsxtra}
\usepackage{latexsym}
\usepackage[mathscr]{eucal}
\usepackage{bm}
\usepackage{rotating}
\usepackage{pgfplots}
\usepackage{hyperref}
\usepackage{emptypage}
\usepackage{enumerate}
\usepackage{enumitem}

\newcommand{\C}{\mathbb{C}}

\newcommand{\Cn}{\mathbb{C}^n}
\newcommand{\D}{\mathbb{D}}
\newcommand{\T}{\mathbb{T}}

\newcommand{\G}{\widetilde{\mathbb{G}}_3}

\newcommand{\gn}{\mathbb{G}_n}
\newcommand{\Gn}{\widetilde{\mathbb{G}}_n}

\newcommand{\Gamn}{\widetilde{\Gamma}_n}

\newcommand{\p}{\Phi_1}
\newcommand{\vp}{\psi}
\newcommand{\M}{\mathcal{M}}

\newcommand{\al}{\alpha}
\newcommand{\be}{\beta}
\newcommand{\lm}{\lambda}
\newcommand{\om}{\omega}
\newcommand{\sg}{\nu}

\newcommand{\q}{\quad}
\newcommand{\qq}{\qquad}
\newcommand{\n}{\lVert}
\newcommand{\imp}{\Rightarrow}

\newcommand{\lf}{\left(}
\newcommand{\ls}{\left\{}
\newcommand{\lt}{\left[}
\newcommand{\rf}{\right)}
\newcommand{\rs}{\right\}}
\newcommand{\rt}{\right]}
\newcommand{\df}{\dfrac}

\newcommand{\la}{\langle}
\newcommand{\ra}{\rangle}
\newcommand{\nj}{{n \choose j}}

\def ~{\hspace{1mm}}

		\def\textmatrix#1&#2\\#3&#4\\{\bigl({#1 \atop #3}\ {#2 \atop #4}\bigr)}
		\def\dispmatrix#1&#2\\#3&#4\\{\left({#1 \atop #3}\ {#2 \atop #4}\right)}

		\newtheorem{defn}{Definition}[section] 

		\newtheorem{lemma}[defn]{Lemma}
		\newtheorem{thm}[defn]{Theorem}
		\newtheorem{rem}[defn]{Remark}
		
\begin{document}
\title[Interpolating functions for $\widetilde{\mathbb G}_n$]
{Interpolating functions for a family of domains related to $\mu$-synthesis}

\author[Samriddho Roy]{Samriddho Roy}
\address[Samriddho Roy]{Tata Institute of Fundamental Research Centre for Applicable Mathematics, Bangalore - 560065, India.} \email{samriddhoroy@gmail.com}




\begin{abstract}
	Assuming the existence of an analytic interpolant mapping a two-point data from the unit disc $\D$ to $\Gn$,
	we describe a class of such interpolating functions where
	\begin{align*}
	\widetilde{\mathbb G}_n := \Bigg\{ (y_1,\dots,y_{n-1}, q)\in
	\C^n :\; q \in \mathbb D, \;  y_j = \be_j + \bar \be_{n-j} q, \; \beta_j \in \mathbb C &\text{ and }\\
	|\beta_j|+ |\beta_{n-j}| < {n \choose j} &\text{ for } j=1,\dots, n-1 \Bigg\}.
	\end{align*}
	We present the connection of $\Gn$ with the  $\mu$-synthesis problem.	
\end{abstract}
\maketitle	
\section{Introduction}
This article is a sequel of \cite{pal-roy 4} and \cite{pal-roy 5}. In \cite{pal-roy 4} the author and Pal introduced a new family of domains, namely 
the \textit{extended symmetrized polydisc}, $\Gn$, where
\begin{align*}
\widetilde{\mathbb G}_n : = \Bigg\{ (y_1,\dots,y_{n-1}, q)\in \C^n
: \;
&|q|<1 , \:  y_j = \be_j + \bar \be_{n-j} q \: \text{ with }\; \\ & |\beta_j|  +  |\beta_{n-j}| < {n \choose j},  1\leq j\leq n-1 \Bigg\},
\end{align*}
and found a variety of new characterizations of the points of 
the \textit{symmetrized polydisc}, $\mathbb G_n$, where
\begin{equation*}
\mathbb G_n =\left\{ \left(\sum_{1\leq i\leq n} z_i,\sum_{1\leq
	i<j\leq n}z_iz_j,\dots,
\prod_{i=1}^n z_i \right): \,|z_i|< 1, i=1,\dots,n \right \}\,.
\end{equation*}	
The symmetrized polydisc is directly associated with the spectral interpolation. For a matrix $A$, the spectral radius $r(A)<1$  if and only if
$\pi_n(\lambda_1,\dots, \lambda_n) \in \mathbb G_n$ (see
\cite{costara1}), where $\lambda_1, \dots , \lambda_n$ are
eigenvalues of $A$ and $\pi_n$ is the symmetrization map on
$\mathbb C^n$ defined by
\[
\pi_n(z_1,\dots, z_n) = \left(\sum_{1\leq i\leq n} z_i,\sum_{1\leq
	i<j\leq n}z_iz_j,\dots, \prod_{i=1}^n z_i \right).
\]

The Schwarz lemma for the symmetrized bidisc $\mathbb G_2$ was presented in \cite{AY-BLMS} (also see \cite{pal-roy 1}). But there was no Schwarz type lemma for $\gn$, $n \geq 3$. Note that $\widetilde{\mathbb G}_2 = \mathbb G_2$ and $\gn \subsetneqq \Gn$ for $n\geq 3$ (see \cite{costara1}, \cite{pal-roy 4}). In \cite{pal-roy 5}, the author and Pal produced a Schwarz lemma for $\Gn$ and $\gn$ and showed that an interpolating function, when exists, may not be unique. The aim of this paper is to describe a class of analytic interpolants when we have such a Schwarz lemma for $\Gn$.

The closure of $\Gn$ is denoted by $\Gamn$. to study complex geometry of $\Gn$ and $\widetilde\Gamma_n$, in \cite{pal-roy 4}, we introduced $(n-1)$ fractional linear transformations $\Phi_1, \dots, \Phi_{n-1}$. Recall that, for $z \in \C$ and $y=(y_1,\dots,y_{n-1},q) \in \Cn $ we define 
\[
\Phi_j(z,y) = \dfrac{{n \choose j}qz-y_j}{y_{n-j}z-{n \choose j}} \q \text{ whenever } y_{n-j}z\neq {n \choose j} \text{ and } y_j y_{n-j}\neq {n \choose j}^2 q.
\]
Theorems $2.5$ and $2.7$ of \cite{pal-roy 4} characterize the points in $\Gn$ and $\Gamn$. The extended symmetrized polydisc is a non-convex but polynomially convex domain for all $n$. Also $\Gn$ is a starlike domain (see \cite{pal-roy 4}).

Let $B_1,\dots, B_k$ be $2 \times 2$ strictly contractive matrices such that $\det B_1= \det B_2 = \cdots = \det B_k$. We define two functions $\pi_{2k+ 1}$ and $\pi_{2k}$ in the following way:
\[
\pi_{2k+1} \lf B_1,\dots, B_k \rf = \Bigg( {n \choose 1} [B_1]_{11},   \dots , {n \choose k} [B_k]_{11},  {n \choose k} [B_k]_{22}, \dots, {n \choose 1} [B_1]_{22} , \det B_1 \Bigg)
\]
and
\begin{align*}
\pi_{2k} \lf B_1,\dots, B_k \rf
=\Bigg( {n \choose 1} [B_1]_{11}, \dots , & {n \choose k-1} [B_{k-1}]_{11}, {n \choose k} \df{\lf [B_k]_{11} + [B_k]_{22} \rf}{2},\\ & \qq {n \choose k-1} [B_{k-1}]_{22}, \dots, {n \choose 1} [B_1]_{22}, \det B_1  \Bigg).
\end{align*}
Then, by Theorem $2.5$ of \cite{pal-roy 4}, we have
$\pi_{2k} \lf B_1,\dots, B_k \rf \in \widetilde{\mathbb G}_{2k} \text{ and } \pi_{2k+1} \lf B_1,\dots, B_k \rf \in \widetilde{\mathbb G}_{2k+1}$. 

For $0 \neq \lm \in \D$ and $y \in \Gn$, the Schwarz lemma for $\Gn$ describes the necessary conditions for the existence of an analytic interpolating function from the unit disc to $\Gn$ mapping the origin to the origin and $\lm$ to $y$. 
In \cite{pal-roy 5}, it is shown that unlike the classical Schwarz Lemma there is no uniqueness statement for the interpolating function in the case of $\Gn$. In the next section we describe the class of such analytic interpolating functions for the Schwarz lemma for $\Gn$. Note that, in \cite{pal-roy 2}, another version of Schwarz lemma for $\Gn$ was presented.
In section $3$ we present the relation of $\Gn$ with the $\mu$-synthesis problem. In section $4$ we take a step towards the Lempert Theory, we find a class of points for which the Carath\'{e}odory pseudo-distance and the Lempert function from the origin coincide.
Note that the techniques that are used here are similar to the paper \cite{awy}.

%
%
%
%

\section{Interpolating functions}
The Schwarz Lemma for $\Gn$ tells us that for $0 \neq \lm_0 \in \D$ and $y^0 =(y_1^0,\dots,y_{n-1}^0,q^0) \in \Gn$, one of the necessary condition for the existence of an analytic interpolating functions $\vp : \D \longrightarrow \Gn$ such that $\vp(0)=(0,\dots,0)$ and $\vp(\lm_0)=y^0$ is for each $j =1 , \dots,n-1 $, $\n \Phi_j(.,y^0) \n_{H^{\infty}} \leq |\lm_0|$ whenever $|y_{n-j}^0|\leq |y_j^0|$. Note that this result generalizes the Schwarz lemma for the symmetrized bidisc $\mathbb G_2$ presented in \cite{AY-BLMS}. In this section we describe the explicit form of such holomorphic interpolating function $\vp$, assuming its existence, whenever $|y_{n-j}^0|\leq |y_j^0|$ and $\n \Phi_j(.,y^0) \n_{H^{\infty}} < |\lm_0|$. Theorem $\ref{interpolating 1}$ is one of the main results of this article, where we portray explicitly how those $\vp$ looks like.

The Schwarz Lemma for $\gn$ and $\Gn$ only talks about the existential criteria of an interpolating function. But it is important to describe the nature of such interpolating functions. In \cite{AY-BLMS}, only an example of an interpolating function with respect to the Schwarz lemma for $\mathbb G_2$ was presented assuming the condition $\n \Phi_1(.,y^0) \n_{H^{\infty}} = |\lm_0|$. In \cite{pal-roy 2}, the author and Pal produced an example of an interpolating function related with an other version of Schwarz lemma for $\Gn$ mentioned therein. Here, instead of a particular example, we found the explicit form of all holomorphic interpolating functions arising in both the Schwarz Lemma for $\gn$ and $\Gn$. Note that, as $\widetilde{\mathbb G}_2 = \mathbb G_2$, Theorem $\ref{Interpolating Function 2}$ provides description of interpolating functions with respect to the Schwarz lemma for $\mathbb G_2$ for the condition $\n \Phi_1(.,y^0) \n_{H^{\infty}} < |\lm_0|$.
We start this section with the following important lemma. \\

Let $Z \in \C^{2\times 2}$ be such that $\n Z \n < 1$ and let $0 \leq \rho < 1$. Let
\begin{equation}{\label{M-rho}}
\mathcal{K}_Z(\rho) =\begin{bmatrix}
[(1-\rho^2 Z^*Z)(1-Z^*Z)^{-1}]_{11} & [(1- \rho^2)(1 - ZZ^*)^{-1}Z]_{21} \\
[(1 - \rho^2)Z^*(1 - ZZ^*)^{-1}]_{12} & [(ZZ^* - \rho^2)(1 - ZZ^*)^{-1}]_{22}
\end{bmatrix}.
\end{equation}

\begin{lemma}\label{prep interpolation 1}
	Let $\lm_0 \in \D \setminus \{0\}$ and $y^0=(y_1^0,\dots,y_{n-1}^0,q^0)\in \Gn$. For a $j\in \left\{1,\dots,n-1\right\}$,
	suppose $y_j^0y_{n-j}^0\neq \nj^2q^0$, $|y_{n-j}^0|\leq |y_j^0|$ and $ \n \Phi_j(.,y^0) \n <|\lm_0|$. For any $\sg > 0$ let
	\begin{equation}\label{Z-sigma-j}
	Z_{\sg,j} = \begin{bmatrix}
	y_j^0/\nj\lm_0 & \sg w_j \\
	\\
	w_j/\sg & y_{n-j}^0/\nj
	\end{bmatrix}
	\end{equation}
	where $w_j^2 = \df{y_j^0y_{n-j}^0 - \nj^2q^0}{\nj^2\lm_0}$. Let $\mathcal{K}_{Z_{\sg,j}}$ be defined by the equation $\eqref{M-rho}$. Also let $\theta_j, \vartheta_j$ be the roots of the equation
	\[z + 1/z = \frac{|\lm_0|}{|y_j^0 y_{n-j}^0 - \nj^2 q^0|}\left( \nj^2 - \frac{|y_j^0|^2}{|\lm_0|^2} -
	|y_{n-j}^0|^2  + \frac{\nj^2 |q^0|^2}{|\lm_0|^2} \right).\] 
	Then, $\n Z_{\sg,j} \n < 1$ if and only if 
	\begin{equation}{\label{range-theta-j}}
	\theta_j < \sg^2 < \vartheta_j.
	\end{equation}
	Also $\mathcal{K}_{Z_{\sg,j}}(|\lm_0|)$ is not positive definite whenever $\sg$ satisfies condition $\eqref{range-theta-j}$.
\end{lemma}

\begin{proof}
	First note that
	\begingroup
	\allowdisplaybreaks
	\begin{align}\label{mat 1}
	1-Z_{\sg,j}^*Z_{\sg,j} =\begin{bmatrix}
	1-\df{|y_j^0|^2}{\nj^2|\lm_0|^2} - \df{|w|^2}{\sg^2}  & -\df{\bar{y}_j^0\sg w}{\nj \bar{\lm}_0}- \df{y_{n-j}^0\bar{w}}{\nj \sg}   \\
	\\
	-\df{y_j^0\sg \bar{w}}{\nj \lm_0}- \df{\bar{y}_{n-j}^0w}{\nj \sg} & 1 - \sg^2|w|^2 -\df{|y_{n-j}^0|^2}{\nj^2}
	\end{bmatrix}.
	\end{align}
	\endgroup
	Then we have 
	\begin{equation}\label{app-2}
	\det(1-Z_{\sg,j}^*Z_{\sg,j})=
	1 - \df{|y_j^0|^2}{\nj^2|\lm_0|^2} - \df{|y_{n-j}^0|^2}{\nj^2} + \df{|q^0|^2}{|\lm_0|^2} - \df{|y_j^0y_{n-j}^0-\nj^2q^0|}{\nj^2|\lm_0|}(\sg + 1/\sg^2).
	\end{equation}
	Thus, $\n Z_{\sg,j} \n  < 1$ if and only if
	\begin{align*}
	& 1 - \df{|y_{n-j}^0|^2}{\nj^2} - \sg^2\df{|y_j^0y_{n-j}^0-\nj^2q^0|}{\nj^2|\lm_0|} > 0  \\
	\text{and} \qq
	& 1 -\df{|y_j^0|^2}{\nj^2|\lm_0|^2} - \df{|y_{n-j}^0|^2}{\nj^2} + \df{|q^0|^2}{|\lm_0|^2} - \df{|y_j^0y_{n-j}^0-\nj^2q^0|}{\nj^2|\lm_0|}(\sg^2 + 1/\sg^2) > 0,
	\end{align*}
	that is, if and only if
	\begin{align*}
	& \sg^2 < \df{\nj^2|\lm_0|\left(1 - \dfrac{|y_{n-j}^0|^2}{\nj^2} \right)}{|y_j^0y_{n-j}^0 - \nj^2q^0|} \\
	\text{and} \qq
	& \sg^2 + 1/\sg^2 < \dfrac{ \nj^2|\lm_0| \lf 1 -\df{|y_j^0|^2}{\nj^2|\lm_0|^2} - \df{|y_{n-j}^0|^2}{\nj^2} + \df{|q^0|^2}{|\lm_0|^2} \rf}{|y_j^0y_{n-j}^0-\nj^2q^0|}.
	\end{align*}
	Let
	\begin{equation}{\label{defn-K}}
	R_j \equiv \df{\nj^2|\lm_0|\left(1 - \dfrac{|y_{n-j}^0|^2}{\nj^2} \right)}{|y_j^0y_{n-j}^0 - \nj^2q^0|} = \df{|\lm_0|(\nj^2 - |y_{n-j}^0|^2)}{|y_j^0y_{n-j}^0 - \nj^2q^0|}.
	\end{equation}
	
	By hypothesis, we have
	\begin{align}\label{D1 to 5}	
	\nonumber & \dfrac{\nj \left|y_j^0 - \bar y_{n-j}^0 q^0 \right| + \left|y_j^0 y_{n-j}^0 - \nj^2 q^0 \right|}{\nj^2 -|y_{n-j}^0|^2} = \n \Phi_j(.,y^0) \n < |\lambda_0| \\
	\imp & \nj \left|\dfrac{y_j^0}{\lm_0} - \bar y_{n-j}^0 \dfrac{q^0}{\lm_0}\right| + \left|\dfrac{y_j^0}{\lm_0} y_{n-j}^0 - \nj^2 \dfrac{q^0}{\lm_0} \right| < \nj^2 -|y_{n-j}^0|^2 .
	\end{align}
	
	Then, using the Theorem $2.5$ of \cite{pal-roy 4} (equivalence of conditions $(4)$ and $(6)$), we have
	\begin{equation}\label{equ for Y-1}
	\nj^2 - \dfrac{|y_j^0|^2}{|\lm_0|^2} - | y_{n-j}^0 |^2 + \nj^2\dfrac{| q^0 |^2}{|\lm_0|^2} > 2\dfrac{\left| y_j^0 y_{n-j}^0 - \nj^2 q^0 \right| }{|\lm_0|}.
	\end{equation}
	Since $|y_{n-j}^0| \leq |y_j^0|$ and $\lm_0 \in \D$, we have $\lf \dfrac{|y_j^0|^2 }{|\lm_0|^2} + |y_{n-j}^0|^2 \rf \geq \lf |y_j^0|^2 + \dfrac{|y_{n-j}^0|^2}{|\lm_0|^2} \rf$.
	Hence
	\begingroup
	\allowdisplaybreaks
	\begin{align}\label{equ for Y-2}
	\nonumber \nj^2 - |y_j^0|^2 - \dfrac{| y_{n-j}^0 |^2}{|\lm_0|^2} + \nj^2\dfrac{| q^0 |^2}{|\lm_0|^2}
	\geq & \nj^2 - \dfrac{|y_j^0|^2}{|\lm_0|^2} - | y_{n-j}^0 |^2 + \nj^2\dfrac{| q^0 |^2}{|\lm_0|^2} \\
	> & 2\dfrac{\left| y_j^0 y_{n-j}^0 - \nj^2 q^0 \right| }{|\lm_0|}.
	\end{align}
	\endgroup
	Let,
	$\qq X_j := \dfrac{|\lm_0|}{|y_j^0 y_{n-j}^0 - \nj^2 q^0|}\left( \nj^2 - |y_j^0|^2 -\dfrac{|y_{n-j}^0|^2}{|\lm_0|^2} + \dfrac{\nj^2 |q^0|^2}{|\lm_0|^2} \right)$ \\
	and $\qq X_{n-j} := \dfrac{|\lm_0|}{|y_j^0 y_{n-j}^0 - \nj^2 q^0|}\left( \nj^2 - \dfrac{|y_j^0|^2}{|\lm_0|^2} -
	|y_{n-j}^0|^2  + \dfrac{\nj^2 |q^0|^2}{|\lm_0|^2} \right) \vspace{0.3cm}$.\\
	Then we have, $X_j > 2$ and $X_{n-j} > 2$.
	Then, $\n Z_{\sg,j} \n  < 1$ if and only if
	$$ \sg^2 < R_j \q \textrm{and} \q \sg^2 + 1/ \sg^2 < X_{n-j}.$$
	Since $|y_{n-j}^0| \leq |y_j^0|$, we have $y_j^0 \neq  \bar{y}_{n-j}^0q^0$. Otherwise, $y_j^0 = \bar{y}_{n-j}^0q^0$ and hence $|y_j^0|=|\bar{y}_{n-j}^0||q^0|<|y_{n-j}^0|$, a contradiction. By hypothesis $\n \Phi_j(.,y^0) \n < |\lm_0|$. Then 
	$$\left|y_j^0y_{n-j}^0 - \nj^2q^0\right| = \left(\nj^2 - |y_{n-j}^0|^2\right) \n \Phi_j(.,y^0) \n < |\lm_0|\left(\nj^2 - |y_{n-j}^0|^2\right).$$

	Hence $ R_j > 1$. 
	Notice that 
	\begin{equation}\label{app-3}
	R_j + 1/R_j - X_{n-j}
	= \dfrac{\nj^2|y_j^0 - \bar{y}_{n-j}^0q^0|^2}{|\lm_0|(\nj^2 - |y_{n-j}^0|^2)|y_j^0y_{n-j}^0 - \nj^2q^0|} > 0.
	\end{equation}
	Hence
	\begin{equation}{\label{cond-K-Y}}
	R_j + 1/R_j > X_{n-j}.
	\end{equation}
	The graph of the function $f(x) = x+ \df{1}{x} $ is the following :
	\begin{figure}[ht!]
		\centering
		\includegraphics[width=70mm]{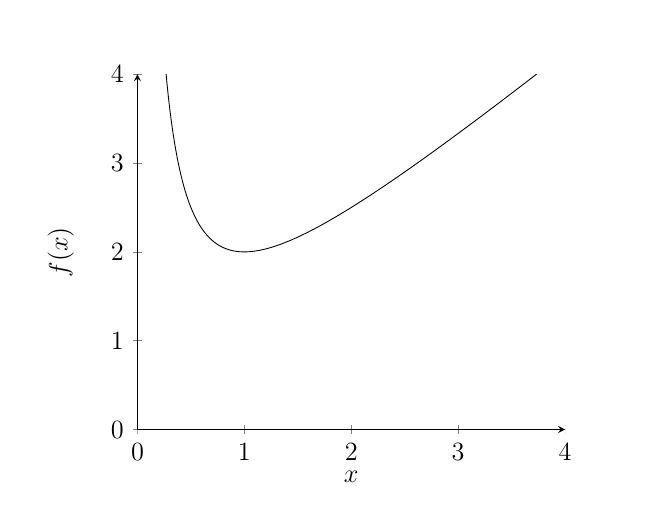}
	\end{figure}\\
	The function $f$ is continuous and has a point of minima at $x=1$ and $f(1)=2$. Since $X_{n-j} > 2$,
	the line $x=X_{n-j}$ intercepts the graph of $f$ twice. Note that the equation
	\[z + 1/z = \frac{|\lm_0|}{|y_j y_{n-j} - \nj^2 q|}\left( \nj^2 - \frac{|y_j|^2}{|\lm_0|^2} -
	|y_{n-j}|^2  + \frac{\nj^2 |q|^2}{|\lm_0|^2} \right)\]
	is same as the equation $z + 1/z =X_{n-j}$.
	Since $\theta_j$ and $\vartheta_j$ are two solutions of the equation $z + 1/z =X_{n-j}$, it is clear from the above figure that $ \xi + 1/ \xi < X_{n-j} $ if and only if $\theta_j < \xi < \vartheta_j $.
	From the fact $\theta_j < 1 < \vartheta_j$, $R_j>1$ and from equation $\eqref{cond-K-Y}$, we have $R_j > \vartheta_j$. 
	
	If $\theta_j < \sg^2 < \vartheta_j$, then $\sg^2 < R_j$ and $\sg^2 + 1/ \sg^2 < X_{n-j}$. Hence $\n Z_{\sg,j}\n < 1$ whenever $\theta_j < \sg^2 < \vartheta_j$. Again  $\n Z_{\sg,j}\n < 1$ implies $\sg^2 + 1/ \sg^2 < X_{n-j}$, which implies $\theta_j < \sg^2 < \vartheta_j$.
	Therefore, $\n Z_{\sg,j}\n < 1$ if and only if $\theta_j < \sg^2 < \vartheta_j$.\\
	
	Next we show that, $\det \mathcal{K}_{Z_{\sg,j}}(|\lm_0|) < 0$ whenever $\theta_j < \sg^2 < \vartheta_j$.
	First note that :
	\begin{align}\label{appendix 1}
	\nonumber & \mathcal{K}_{Z_{\sg,j}}(|\lm_0|) \\
	&=\begin{bmatrix}
	[(1-|\lm_0|^2 Z_{\sg,j}^*Z_{\sg,j})(1-Z_{\sg,j}^*Z_{\sg,j})^{-1}]_{11} & [(1- |\lm_0|^2)(1 - Z_{\sg,j} Z_{\sg,j}^*)^{-1}Z_{\sg,j}]_{21} \\
	\\
	[(1 - |\lm_0|^2)Z_{\sg,j}^*(1 - Z_{\sg,j}Z_{\sg,j}^*)^{-1}]_{12} & [(Z_{\sg,j}Z_{\sg,j}^* - |\lm_0|^2)(1 - Z_{\sg,j}Z_{\sg,j}^*)^{-1}]_{22}
	\end{bmatrix}.
	\end{align}
	Then by simple calculations, we have (see appendix)
	\begingroup
	\allowdisplaybreaks
	\begin{align}\label{appendix 1.1}
	\nonumber & \mathcal{K}_{Z_{\sg,j}}(|\lm_0|)\det (1-Z_{\sg,j}^*Z_{\sg,j}) = \\
	& \begin{bmatrix}
	1 - \dfrac{|y_j^0|^2}{\nj^2} - \dfrac{|y_{n-j}^0|^2}{\nj^2} + |q^0|^2  & \q &
	(1 - |\lm_0|^2)\Big( \dfrac{w}{\sg} + \dfrac{q^0}{\lm_0}\sg \bar{w} \Big) \\
	- \dfrac{|y_j^0y_{n-j}^0 - \nj^2q^0|}{\nj^2}
	\Big( \dfrac{|\lm_0|}{\sg^2} + \dfrac{\sg^2}{|\lm_0|} \Big) & \q & \q \\
	\\
	\q & \q &
	- |\lm_0|^2 + \dfrac{|y_j^0|^2}{\nj^2} + \dfrac{|y_{n-j}^0|^2}{\nj^2} - \dfrac{|q^0|^2}{|\lm_0|^2} \\
	(1 - |\lm_0|^2)\Big( \dfrac{\bar{w}}{\sg} + \dfrac{\bar{q}^0}{\bar{\lm}_0}\sg w \Big) & \q  & - \dfrac{|y_j^0y_{n-j}^0 - \nj^2q^0|}{\nj^2} \Big( \sg^2|\lm_0| + \dfrac{1}{\sg^2|\lm_0|} \Big)\\
	\end{bmatrix}.
	\end{align}
	\endgroup
	Now by a straight forward calculation 
	\begin{equation}\label{app-6}
	\det(\mathcal{K}_{Z_{\sg,j}}(|\lm_0|)\det (1-Z_{\sg,j}^*Z_{\sg,j})) = -(l_{\nu,j} - k_j)(l_{\nu,j} - k_{n-j}),
	\end{equation}
	where
	\begin{align*}
	l_{\nu,j} & = |y_j^0y_{n-j}^0 - \nj^2q^0|\Big( \sg^2 + \df{1}{\sg^2} \Big), \\
	k_j & = |y_j^0y_{n-j}^0 - \nj^2q^0|X_j.
	\end{align*}
	Since $|y_{n-j}^0| \leq |y_j^0|$,
	$$ X_j - X_{n-j} = \df{1 - |\lm_0|^2}{|(y_j^0y_{n-j}^0 - \nj^2q^0)\lm_0|}\Big( |y_j^0|^2 - |y_{n-j}^0|^2 \Big) \geq 0 $$
	and consequently $k_{n-j} \leq k_j$.
	If $\theta_j < \sg^2 < \vartheta_j$ then $\sg^2 + \df{1}{\sg^2} < X_{n-j}$ and hence $l_{\nu,j}< |y_j^0 y_{n-j}^0 - \nj^2q^0| X_{n-j} = k_{n-j} \leq k_j$.
	Thus $\det \mathcal{K}_{Z_{\sg,j}}(|\lm_0|) < 0$ when $\theta_j < \sg^2 < \vartheta_j$.
\end{proof}

The next theorem is one of the main results of this article, it describes the class of analytic interpolating functions from the unit disc to $\Gn$. We use two lemmas, Lemma $3.1$ and Lemma $3.2$, from \cite{awy}.

Let $Z$ be a strict $2\times 2$ matrix contraction and $\al \in \C^2  \setminus \{0\}$. A matricial M\"{o}bius transformation $\M_Z$ is defined as
\[
\M_Z(X) = (1- ZZ^*)^{-\frac{1}{2}} (X-Z)(1- Z^*X)^{-1} (1-Z^*Z)^{\frac{1}{2}}\;, \quad X\in \C^{2\times 2} \text{ and } \n X \n <1.
\]
Then $\M_Z$ is an automorphism of the close unit ball of $\C^{2 \times 2}$ which maps $Z$ to the zero matrix and $\M_{Z}^{-1}=\M_{-Z}$.
Also let
\begin{align}{\label{u-v}}
	& u_{Z}(\al) = (1-ZZ^*)^{-\frac{1}{2}}(\al_1Ze_1 + \al_2e_2),\\
	\nonumber & v_{Z}(\al) = -(1-Z^*Z)^{-\frac{1}{2}}(\al_1e_1 + \al_2Z^*e_2)
\end{align}
where $\{e_1 , e_2\}$ is the standard basis of $\C^2$.
Then for any $2\times 2$ matrix contraction $X$, by Lemma $3.1$ of \cite{awy},  $[\M_{-Z} (X)]_{22} = 0$ if and only if there exists $\al \in \C^2  \setminus \{0\} $ such that $ X^*u_{Z}(\al) = v_{Z}(\al).$

\begin{rem}
	Let $\lm_0 \in \D \setminus \{0\}$ and $y^0=(y_1^0,\dots,y_{n-1}^0,q^0)\in \Gn$. For a $j\in \left\{1,\dots,n-1\right\}$,
	suppose $y_j^0y_{n-j}^0\neq \nj^2q^0$, $|y_{n-j}^0|\leq |y_j^0|$ and $ \n \Phi_j(.,y^0) \n <|\lm_0|$. Then
	for any $\sg$ satisfying condition $\eqref{range-theta-j}$, using the last lemma and Lemma $3.1$ of \cite{awy} together,
	there exists $X\in \C^{2\times 2}$ such that $\n X \n \leq |\lm_0|$ and $[\M_{-Z_{\sg,j}} (X)]_{22} = 0$, where $Z_{\sg,j}$ is given by $\eqref{Z-sigma-j}$.
\end{rem}

Recall that the \textit{Schur class} of type $2 \times 2$ is the set of analytic functions $F$ on $\mathbb{D}$ with values in the space $\mathbb{C}^{2\times 2}$ such that $\lVert F(\lambda) \lVert \leq 1$ for all $\lambda \in \mathbb{D}$. Also we say that $F\in \mathcal{S}_{2 \times 2}$ if $\lVert F(\lambda) \lVert < 1$ for all $\lambda \in \mathbb{D}$.


\begin{thm}\label{interpolating 1}
	Let $\lm_0 \in \D \setminus \{0\}$ and $y^0=(y_1^0,\dots,y_{n-1}^0,q^0)\in \Gn$ and
	suppose $y_j^0y_{n-j}^0\neq \nj^2q^0$, $|y_{n-j}^0|\leq |y_j^0|$ and $ \n \Phi_j(.,y^0) \n <|\lm_0|$ for each $j =1 , \dots,\left[\frac{n}{2}\right] $. Suppose
	$\vp : \D \longrightarrow  \Gn $ is an analytic function such that $\; \vp(0) = (0,\dots,0) $ and $\; \vp(\lm_0) = y^0 $. Then
	\[
	\vp(\lm) =
	\begin{cases}
	\pi_{2[\frac{n}{2}]+1} \lf F_1(\lm),\dots, F_{[\frac{n}{2}]}(\lm) \rf \q & \text{if } n \text{ is odd}\\
	\pi_{2[\frac{n}{2}]} \lf F_1(\lm),\dots, F_{[\frac{n}{2}]}(\lm) \rf \q & \text{if } n \text{ is even},
	\end{cases}
	\]
	where for each $j =1 , \dots,\left[\frac{n}{2}\right] $,
	\begin{equation}{\label{F j}}
	F_j(\lm)
	= \mathcal{M}_{-Z_{\sg,j}} \Big( (BQ_j)(\lm) \Big)
	\begin{bmatrix}
	\lm & 0\\
	0 & 1
	\end{bmatrix},
	\end{equation}
	$Z_{\sg,j}$ is given by $\eqref{Z-sigma-j}$ and $\sg$ satisfies $\eqref{range-theta-j}$, $Q_j$ is a $2 \times 2$ Schur function such that
	\begin{equation}{\label{formulae-Q j}}
	Q_j(0)^*\bar{\lm}_0 u_{Z_{\sg,j}}(\alpha_j) = v_{Z_{\sg,j}}(\alpha_j),
	\end{equation}
	where $ u_{Z_{\sg,j}}(\alpha_j), v_{Z_{\sg,j}}(\alpha_j)$ are given by equation $\eqref{u-v}$ and $\alpha_j \in \C^2 \setminus  \{0\}$ satisfies
	\begin{equation}{\label{condition-alpha j}}
	\la \mathcal{K}_{Z_{\sg,j}}(|\lm_0|)\alpha_j,\alpha_j \ra \leq 0.
	\end{equation}	
\end{thm}

\begin{proof}
	Suppose $\vp : \D \longrightarrow  \Gn $ is an analytic function such that $\vp(0)= (0,\dots,0)$ and $\vp(\lm_0) = (y_1^0,\dots,y_{n-1}^0,q^0)$. By Fatou's Lemma, $\vp$ has
	radial limits almost everywhere on $\mathbb{T}$. We denote the radial limit function
	of $\vp$ by $\tilde{\vp}$  which maps $\mathbb{T}$ almost everywhere to $\Gamn$.
	Let
	\[h_j(\lm) =\df{\vp_j(\lm)\vp_{n-j}(\lm)}{\nj^2} - \vp_n(\lm) \q \text{for } \lm \in \D .\]
	Then $h_j$ is a bounded analytic function on $\D$ (eventually a Schur function by condition $(4)$ of Theorem 2.5 of \cite{pal-roy 4}).
	By inner-outer factorization, there exist $f_j,g_j \in H^{\infty}$ such that
	$$
	h_j(\lm) = f_j(\lm)g_j(\lm) \q \text{for } \lm \in \D ,
	$$
	with $|\tilde{f_j}|=|\tilde{g_j}|$ almost everywhere on $\mathbb{T}$. As $f_jg_j(0)=0$, without loss of generality assume $g_j(0)=0$. Consider
	\begin{equation}
	F_j(\lm) = \begin{bmatrix}
	\dfrac{\vp_j(\lm)}{\nj} & f_j(\lm) \\
	\\
	g_j(\lm) & \dfrac{\vp_{n-j}(\lm)}{\nj}
	\end{bmatrix}
	\q \text{for } \lm \in \D .
	\end{equation}
	Clearly $\pi_n \big(F_1,\dots, F_{[\frac{n}{2}]}\big)= \vp$.
	Since $f_jg_j = \dfrac{\vp_j\vp_{n-j}}{\nj^2} - \vp_n$, we have $\det F_j = \vp_n$. Note that,
	\begin{equation*}
	1-F_j^*F_j =
	\begin{bmatrix}
	1-\dfrac{|\vp_j|^2}{\nj^2} - |g_j|^2 & & -f_j \dfrac{\bar{\vp}_j}{\nj} - \bar g_j \dfrac{\vp_{n-j}}{\nj} \\
	\\
	-\bar f_j\dfrac{\vp_j}{\nj} - g_j \dfrac{\bar{\vp}_{n-j}}{\nj} & & 1-\dfrac{|\vp_{n-j}|^2}{\nj^2} - |f_j|^2
	\end{bmatrix}
	\end{equation*}
	and
	$$ \det (1-F_j^*F_j) = 1-\df{|\vp_j|^2}{\nj^2} -\df{|\vp_{n-j}|^2}{\nj^2} +|\vp_n|^2 - |f|^2 - |g|^2 .$$
	Since $|\tilde{f}|=|\tilde{g}|$ almost everywhere on $\mathbb{T}$, we have
	$$
	[ 1- \tilde F_j^* \tilde F_j ]_{11}= 1-\df{|\tilde{\vp}_j|^2}{\nj^2} - \left|\df{\tilde{\vp}_j\tilde{\vp}_{n-j}}{\nj^2} - \tilde{\vp}_n\right|
	\, , \q
	[ 1- \tilde F_j^* \tilde F_j ]_{22} = 1-\df{|\tilde{\vp}_{n-j}|^2}{\nj^2} - \left|\df{\tilde{\vp}_j\tilde{\vp}_{n-j}}{\nj^2} - \tilde{\vp}_n \right|
	$$
	almost everywhere on $\mathbb{T}$ and
	$$ \det (1- \tilde F_j^* \tilde F_j) = 1-\df{|\tilde{\vp}_j|^2}{\nj^2} -\df{|\tilde{\vp}_{n-j}|^2}{\nj^2} + |\tilde{\vp}_n|^2 - \Big|\df{\tilde{\vp}_j\tilde{\vp}_{n-j}}{\nj^2} - \tilde{\vp}_n\Big| $$
	almost everywhere on $\mathbb{T}$. Then, using Theorem $2.5$ of \cite{pal-roy 4} (by conditions $(3)$, $(3')$ and $(5)$), we have
	\[
	\lt 1- \tilde F_j^* \tilde F_j \rt_{11}\geq 0, \; \lt 1-\tilde F_j^* \tilde F_j \rt_{22}\geq 0  \text{ and } \, \det (1- \tilde F_j^* \tilde F_j) \geq 0 \, \text{ almost everywhere on } \mathbb{T}.
	\]
	Thus $\n \tilde F_j \n \leq 1$ almost everywhere on $\mathbb{T}$. 
	Note that if $F_j$ is constant then the function $(\vp_j, \vp_{n-j}, \vp_n)$ is also constant. Since $\vp(0)= (0,\dots,0)$, $\vp(\lm_0)= (y_1^0,\dots,y_{n-1}^0,q^0)$ and $y_j^0y_{n-j}^0 \neq \nj^2q^0$, we conclude $F_j$ is non-constant. Therefore by maximum modulus principle, $\n F_j(\lm) \n < 1$ for all $\lm \in \D$. Hence $ F_j \in S_{2 \times 2}$.
	
	Thus to complete the proof we need to show that $F_j$ can be written in the form $\eqref{F j}$ for some $\sg,\al_j$ and $Q_j$ satisfying $\eqref{range-theta-j},\eqref{condition-alpha j}$ and $\eqref{formulae-Q j}$ respectively.
	Since $y_j^0y_{n-j}^0 \neq \nj^2q^0$, we have $f_jg_j(\lm_0) = \dfrac{y_j^0 y_{n-j}^0}{\nj^2} -q^0 \neq 0$ and hence both $f_j(\lm_0)$ and $g_j(\lm_0)$ are nonzero.
	Now consider $\sg = \df{f_j(\lm_0)}{w_j}$ where $w_j^2 = \df{y_j^0y_{n-j}^0 - \nj^2q^0}{\nj^2\lm_0}$. Then $g_j(\lm_0)= \dfrac{\lm_0 w_j^2}{f_j(\lm_0)}=\df{\lm_0w_j}{\sg}$. Hence
	$$ F_j(\lm_0)
	=\begin{bmatrix}
	y_j^0/\nj &  \sg w_j\\
	\lm_0\sg^{-1}w_j & y_{n-j}^0/\nj
	\end{bmatrix}.$$
	We may assume that $\sg > 0$ (otherwise we can replace $F_j$ by $U^*F_jU$ for a suitable diagonal unitary matrix $U$). Consider
	$$ G_j(\lm)= F_j(\lm)\begin{bmatrix}
	\lm^{-1} & 0 \\
	0 & 1
	\end{bmatrix} \q \text{for all } \lm \in \D .$$
	Then we have $G_j \in S_{2 \times 2}$ and
	$$ [G_j(0)]_{22}=[F_j(0)]_{22} = 0 \q \text{and} \q
	G_j(\lm_0)
	= \begin{bmatrix}
	y_j^0/\nj\lm_0 &  \sg w_j\\
	\sg^{-1}w_j & y_{n-j}^0/\nj
	\end{bmatrix}
	= Z_{\sg,j}.
	$$
	Then $F_j(\lm)= G_j(\lm)\begin{bmatrix}
	\lm & 0 \\
	0 & 1
	\end{bmatrix}$ for all $\lm \in \D$. We have already proved that $\n Z_{\sg,j} \n < 1$
	if and only if $\theta_j < \sg^2 < \vartheta_j$. Since $G_j \in S_{2 \times 2}$, $\n Z_{\sg,j} \n = \n G_j(\lm_0)\n <1 $ and hence $\theta_j < \sg^2 < \vartheta_j$. That is, condition $\eqref{range-theta-j}$ holds. Again since $G_j \in S_{2 \times 2}$,  $[G_j(0)]_{22} = 0$ and $G_j(\lm_0) = Z_{\sg,j}$, by part-$(2)$ of Lemma $3.2$ of \cite{awy}, there exists some $\al_j \in \C^2 \setminus  \{0\}$ such that $\la \mathcal{K}_{Z_{\sg,j}}(|\lm_0|)\al_j,\al_j \ra \leq 0$ and a Schur function $Q_j$ such that $Q_j(0)^*\bar{\lm}_0 u_{Z_{\sg,j}}(\alpha_j) =  v_{Z_{\sg,j}}(\alpha_j)$ and
	$$G_j = \mathcal{M}_{-Z_{\sg,j}} \circ BQ_j,$$
	where $ u_{Z_{\sg,j}}(\alpha_j),  v_{Z_{\sg,j}}(\alpha_j)$ are given by equation $\eqref{u-v}$. Thus the conditions $\eqref{condition-alpha j}$ and $\eqref{formulae-Q j}$ are satisfied. Hence
	$F_j(\lm)= \mathcal{M}_{-Z_{\sg,j}} \lf BQ_j(\lm) \rf
	\begin{bmatrix}
	\lm & 0 \\
	0 & 1
	\end{bmatrix}$
	that is, $F_j$ can be written in the form $\eqref{F j}$. Hence the proof is complete.
\end{proof}

\begin{rem}
	For $0 \neq\lm \in \D $ and $x \in \gn$, any analytic function $\phi : \D \longrightarrow \gn$ such that $\phi(0)=(0,\cdots, 0)$ and $\phi(\lm)=x$ then $\phi$ can be viewed as a map $\phi : \D \longrightarrow \Gn$ (as $\gn \subset \Gn$) and by applying Theorem $\ref{interpolating 1}$ we can describe all such function $\phi$.
\end{rem}

The Schwarz lemma for $\Gn$ says that, for some special case, when the target point $y^0$ belongs to a particular subset $\mathcal J_n$ of $\Gn$, the conditions which are necessary for the existence of the required interpolating function also become sufficient. Recall that when $n$ is odd,
\begin{align*}
\mathcal J_n = \Big\{ \lf y_1, \dots, y_{n-1}, q \rf \in \Gn : y_j = \dfrac{\nj}{n} y_1,\; y_{n-j} = \dfrac{\nj}{n}  y_{n-1}, \; 2 \leq j \leq \left[\frac{n}{2}\right] \Big\}
\end{align*}
and for $n$ even,
\begin{align*}
\mathcal J_n = \Big\{ \lf y_1, \dots, y_{n-1}, q \rf \in \Gn : y_{[\frac{n}{2} ]}= &{n \choose [\frac{n}{2} ]} \df{y_1 + y_{n-1}}{2n}\,,\;  y_j = \dfrac{\nj}{n} y_1 \,, \\ & y_{n-j} = \dfrac{\nj}{n}  y_{n-1},  \; 2 \leq j \leq \left[\frac{n}{2}\right] -1  \Big\}.
\end{align*}
Suppose $\lm_0 \in \D \setminus \{0\}$, $y^0=(y_1^0,\dots,y_{n-1}^0,q^0)\in \mathcal J_n$ such that $\n y_{n-1}^0 \n \leq \n y_1^0 \n$ and  $ \n \p(.,y^0) \n <|\lm_0|$. Then the Schwarz Lemma for $\Gn$ guarantees the existence of an analytic function
$\vp : \D \longrightarrow \mathcal J_n \subset \Gn $ such that  $\; \vp(0) = (0,\dots,0) $ and $\; \vp(\lm_0) = y^0 $ (see \cite{pal-roy 5}). \\

In Theorem $\ref{interpolating 1}$, assuming the existence of an analytic interpolating function $\vp: \D \longrightarrow \Gn$ in the Schwarz Lemma for $\Gn$, we explicitly describe the function $\vp$. The next theorem deals with the converse part when the target point $y^0 \in \mathcal J_n$. We show that if $\vp$ is of the form similar to that mentioned in Theorem $\ref{interpolating 1}$, then $\vp$ is an analytic interpolating function from unit disc to $\mathcal J_n \subset \Gn$ mapping the origin to the origin and $\lm_0$ to $y^0$.
The following lemma is just a particular case of Lemma \ref{prep interpolation 1} and can be proved similarly.

\begin{lemma}\label{prep interpolating}
	Let $\lm_0 \in \D \setminus \{0\}$ and $y^0=(y_1^0,\dots,y_{n-1}^0,q^0)\in \mathcal J_n$ and
	suppose that $y_1^0y_{n-1}^0\neq n^2q^0$, $|y_{n-1}^0|\leq |y_1^0|$ and $ \n \p(.,y^0) \n <|\lm_0|$. For any $\sg > 0$ let
	\begin{equation}{\label{Z(sg)}}
	Z_{\sg} = \begin{bmatrix}
	y_1^0/n\lm_0 & \sg w \\
	\\
	\sg^{-1}w & y_{n-1}^0/n
	\end{bmatrix}
	\end{equation}
	where $w^2 = \df{y_1^0y_{n-1}^0 - n^2q^0}{n^2\lm_0}$. Let $\mathcal{K}_{Z_{\sg}}$ be defined by the equation $\eqref{M-rho}$. Also let $\theta_1, \vartheta_1$ be the roots of the equation
	\[z + 1/z = \frac{|\lm_0|}{|y_1^0 y_{n-1}^0 - n^2 q^0|}\left( n^2 - \frac{|y_1^0|^2}{|\lm_0|^2} -
	|y_{n-1}^0|^2  + \frac{n^2 |q^0|^2}{|\lm_0|^2} \right).\] Then $\n Z_{\sg} \n < 1$ if and only if 
	\begin{equation}{\label{range-sigma}}
	\theta_1 < \sg^2 < \vartheta_1.
	\end{equation}
	Also $\mathcal{K}_{Z_{\sg}}(|\lm_0|)$ is not positive definite for any $\sg$ satisfying $\eqref{range-sigma}$.
\end{lemma}

For a $2 \times 2$ contractive matrix $B$, let $\widehat\pi_n$ be the following map
\begin{align*}
\widehat\pi_n (B) =  
\begin{cases}
\pi_{2[\frac{n}{2}]+1} \lf B,\dots, B \rf \q & \text{if } n \text{ is odd}\\
\pi_{2[\frac{n}{2}]} \lf B,\dots, B \rf \q & \text{if } n \text{ is even}.
\end{cases}
\end{align*}
Then, for $F \in S^{2 \times 2}$, it is clear that $\widehat\pi_n \circ F $ is a function from $\D$ to $\mathcal J_n$.

\begin{thm}\label{Interpolating Function}
	Let $\lm_0 \in \D \setminus \{0\}$ and $y^0=(y_1^0,\dots,y_{n-1}^0,q^0)\in \mathcal J_n$ and
	suppose that $y_1^0y_{n-1}^0\neq n^2q^0$, $|y_{n-1}^0|\leq |y_1^0|$ and $ \n \p(.,y^0) \n <|\lm_0|$. Let $\mathcal{R}$ be the set of analytic functions
	$\vp : \D \longrightarrow \mathcal J_n \subset \Gn $ satisfying  $\; \vp(0) = (0,\dots,0) $ and $\; \vp(\lm_0) = y^0 $.
	The function $\widehat\pi_n \circ F$ belongs to $\mathcal{R}$ where
	\begin{equation}{\label{F}}
	F(\lm)
	= \mathcal{M}_{-Z_{\sg}} \Big( (BQ)(\lm) \Big)
	\begin{bmatrix}
	\lm & 0\\
	0 & 1
	\end{bmatrix}
	\end{equation}
	where $Z_{\sg}$ is given by $\eqref{Z(sg)}$, $\sg$ satisfies $\eqref{range-sigma}$, $Q$ is a $2 \times 2$ Schur function such that
	\begin{equation}{\label{formulae-Q}}
	Q(0)^*\bar{\lm}_0 u_{Z_{\sg}}(\alpha) = v_{Z_{\sg}}(\alpha),
	\end{equation}
	where $ u_{Z_{\sg}}(\alpha), v_{Z_{\sg}}(\alpha)$ are given by equation $\eqref{u-v}$ and $\alpha \in \C^2 \setminus  \{0\}$ satisfies
	\begin{equation}{\label{condition-alpha}}
	\la \mathcal{K}_{Z_{\sg}}(|\lm_0|)\alpha,\alpha \ra \leq 0.
	\end{equation}
	Conversely, every function in $\mathcal{R}$ is of the form $\widehat\pi_n \circ F$ where $F$ and $Z_{\sg}$ are given by equation $\eqref{F}$ and $\eqref{Z(sg)}$, for some choice of $\sg,\alpha$ and $Q$
	satisfying the conditions $\eqref{range-sigma},\eqref{condition-alpha}$ and $\eqref{formulae-Q}$, respectively.
\end{thm}

\begin{proof}
	Suppose $\sg$ satisfies $\eqref{range-sigma}$, that is, $\theta_1 < \sg^2 < \vartheta_1$, where $\theta_1, \vartheta_1$ are the roots of the equation
	\[z + 1/z = \frac{|\lm_0|}{|y_1^0 y_{n-1}^0 - n^2 q^0|}\left( n^2 - \frac{|y_1^0|^2}{|\lm_0|^2} -
	|y_{n-1}^0|^2  + \frac{n^2 |q^0|^2}{|\lm_0|^2} \right) .\]
	Then by Lemma \ref{prep interpolating}, $\n Z_{\sg} \n < 1$ and $ \mathcal{K}_{Z_{\sg}}(|\lm_0|)$ is not positive definite. Further, suppose $\al \in \C^2 \setminus \ls (0,0) \rs$  such that
	$\la \mathcal{K}_{Z_{\sg}}(|\lm_0|)\al,\al \ra \leq 0$ and also suppose $Q$ is a Schur function such that $Q(0)^*\bar{\lm}_0u_{Z_{\sg}}(\al) = v_{Z_{\sg}}(\al)$.
	Then using part $(2)$ of Lemma $3.2$ of \cite{awy}, there exists a function $G \in S_{2 \times 2}$ such that
	\[ G = \mathcal{M}_{-Z_{\sg}}\circ (BQ)  \qq \text{and} \qq [G(0)]_{22}=0,\;
	G(\lm_0)=Z_{\sg}.\]
	Thus the function $F$, given by equation $\eqref{F}$, is
	\[
	F(\lm) = G(\lm)\begin{bmatrix}
	\lm & 0 \\
	0 & 1
	\end{bmatrix} \q \text{for } \lm \in \D.
	\]
	It is clear that $F \in S_{2 \times 2}$ and $F$ satisfies the following
	$$
	F(0) = \begin{bmatrix}
	0 & * \\
	0 & 0 \\
	\end{bmatrix},
    \q
	F(\lm_0) = \begin{bmatrix}
	y_1^0 /n & \sg w \\
	\lm_0 \sg^{-1}w & y_{n-1}^0 /n \\
	\end{bmatrix}.
	$$
	Thus the function $\vp = \widehat\pi_n \circ F$ is analytic from $\D$ to $\mathcal J_n$ such that $\vp(0) = (0,\dots,0)$ and $\vp(\lm_0) = (y_1^0,\dots,y_{n-1}^0,q^0)=y^0$. Hence $\vp \in \mathcal{R}$.\\
	
	   The proof of the converse part is similar to the proof of Theorem \ref{interpolating 1}.
	   Suppose $\vp \in \mathcal{R}$. Denote the radial limit function
	   of $\vp$ by $\tilde{\vp}$ which, by Fatou's Lemma, maps $\mathbb{T}$ almost everywhere to $\Gamn$.
	   Consider the bounded analytic function
	   \[h(\lm) =\df{\vp_1(\lm)\vp_{n-1}(\lm)}{n^2} - \vp_n(\lm) \q \text{for } \lm \in \D ,\]
	   and write the inner-outer factorization as
	   $
	   h(\lm) = f(\lm)g(\lm) \q \text{for } \lm \in \D,
	   $
	   where $f,g \in H^{\infty}$ such that $|\tilde{f}|=|\tilde{g}|$ almost everywhere on $\mathbb{T}$. 
	   Consider
	   $F(\lm) = \begin{bmatrix}
	   \vp_1(\lm)/n & f(\lm) \\
	   \\
	   g(\lm) & \vp_{n-1}(\lm)/n
	   \end{bmatrix}$, $\lm \in \D$.
	   Clearly $\widehat\pi_n \circ F = \vp$. Using similar method, as shown in proof of Theorem \ref{interpolating 1}, we can show that $ F \in S_{2 \times 2}$.

	   Again using the method similar to that in proof of Theorem \ref{interpolating 1}, we show that $F$ is of the desired form.
	   Consider $\sg = \df{f(\lm_0)}{w}$, then $g(\lm_0)= \df{\lm_0w}{\sg}$.
	   We may assume that $\sg > 0$. 
	   Also consider
	   $ G(\lm)= F(\lm)\begin{bmatrix}
	   \lm^{-1} & 0 \\
	   0 & 1
	   \end{bmatrix}, \;  \lm \in \D .$
	   We may assume that $\sg > 0$.
	   Then we have $G \in S_{2 \times 2}$, $[G(0)]_{22}= 0$ and $G(\lm_0) =Z_{\sg}$.
	   Then $\n Z_{\sg} \n = \n G(\lm_0)\n <1 $ and hence $\theta_1 < \sg^2 < \vartheta_1$.  By part-$(2)$ of Lemma $3.2$ of \cite{awy}, there exists some $\al \in \C^2 \setminus  \{0\}$ such that $\la \mathcal{K}_{Z_{\sg}}(|\lm_0|)\al,\al \ra \leq 0$ and a Schur function $Q$ such that $Q(0)^*\bar{\lm}_0u_{Z_{\sg}}(\al) = v_{Z_{\sg}}(\al)$ and
	   $G = \mathcal{M}_{-Z_{\sg}} \circ BQ,$
	   where $u_{Z_{\sg}}(\al), v_{Z_{\sg}}(\al)$ are given by equation $\eqref{u-v}$. Thus the conditions $\eqref{range-sigma}$, $\eqref{condition-alpha}$ and $\eqref{formulae-Q}$ are satisfied and $F$ can be written as
	   $F(\lm)= \mathcal{M}_{-Z_{\sg}} \lf BQ(\lm) \rf
	   \begin{bmatrix}
	   \lm & 0 \\
	   0 & 1
	   \end{bmatrix}$.
\end{proof}


Theorem $\ref{Interpolating Function}$ is true for $n \geq 3$, because of the fact that $\mathcal J_n$ is defined only for $n \geq 3$. We now present a result similar to Theorem $\ref{Interpolating Function}$ for the case $n=2$.
Note that the conditions described in Schwarz lemma for $\Gn$ in \cite{pal-roy 5} are also sufficient for the existence of an analytic interpolating function from the unit disc to the symmetrized bidisc $\mathbb G_2$ (also see \cite{AY-BLMS} and \cite{pal-roy 1}). The next theorem describes and also provide example of the interpolating function related to Schwarz lemma for $\mathbb G_2$ for the condition that is mutually exclusive to that in Theorem $1.4$ of \cite{AY-BLMS}. 

\begin{thm}\label{Interpolating Function 2}
	Let $\lm_0 \in \D \setminus \{0\}$ and $y^0= (s_0,p_0) \in \mathbb G_2$ be such that ${s_0}^2\neq 4 p_0$ and 
	\[ \dfrac{2|s_0 - \bar s_0 p_0| + |{s_0}^2 - 4p_0|}{4 - |s_0|^2} <|\lm_0|.\]
    Let  
    $\theta, \vartheta$ be the roots of the equation
    \[z + 1/z = \frac{ 4|\lm_0|^2 - |s_0|^2 -
    	|s_0|^2 |\lm_0|^2  + 4 |p_0|^2}{|\lm_0||{s_0}^2 - 4p_0|}\] and for $\sg>0$
	\begin{equation}\label{Z sg 3}
	Z_{\sg} = \begin{bmatrix}
	s_0/2\lm_0 & \sg w \\
	\\
	\sg^{-1}w & s_0/2
	\end{bmatrix}
	\end{equation}
	where $w^2 = \df{{s_0}^2 - 4p_0}{4\lm_0}$.
	Suppose a function $\vp$ is given by 
	\[\vp(\lm) = \Big( [F(\lm)]_{11} + [F(\lm)]_{22} , \det F(\lm) \Big) \q \text{for } \lm \in \D,\] 
	where $F$ is given by equation $\eqref{F}$ with $Z_{\sg}$ as in equation \eqref{Z sg 3}, for some $\sg$ satisfying $\theta < \sg^2 < \vartheta $
	and $\alpha$ and $Q$ satisfying the conditions $\eqref{condition-alpha}$ and $\eqref{formulae-Q}$, respectively. Then $\vp$ is an analytic interpolating function from the unit disc to $\mathbb G_2$ satisfying  $\; \vp(0) = (0,0) $ and $\; \vp(\lm_0) = y^0 $.
	 
	Conversely, any analytic functions
	$\vp : \D \longrightarrow  \mathbb G_2$ satisfying  $\; \vp(0) = (0,0) $ and $\; \vp(\lm_0) = y^0 $ is of the form $\vp(\lm) = \left( [F(\lm)]_{11} + [F(\lm)]_{22} , \det F(\lm) \right)$ where $F$, $Z_{\sg}$ are given by $\eqref{F}$, \eqref{Z sg 3}, $\sg$ satisfies $\theta < \sg^2 < \vartheta $
	and $\alpha$, $Q$ are as in the conditions $\eqref{condition-alpha}$, $\eqref{formulae-Q}$, respectively.	 
\end{thm}

\begin{proof}
	First note that $\widetilde{\mathbb G}_2 = \mathbb G_2$ and $\widehat{\pi}_2 \circ F = \left( [F(\lm)]_{11} + [F(\lm)]_{22} , \det F(\lm) \right)$. For any $F \in S_{2 \times 2}$ the function $\widehat\pi_2 \circ F$ is analytic from $\D$ to $\mathbb G_2$. It is clear, by the definition of $\p$ for the case $n=2,$ that 
	\[ \n \p(.,y^0) \n = \dfrac{2|s_0 - \bar s_0 p_0| + |{s_0}^2 - 4p_0|}{4 - |s_0|^2}  <|\lm_0|.\]
	Then the Proof of both the parts are exactly similar to that of Theorem $\ref{Interpolating Function}$ upon substituting $q^0 = p_0$, $y_1^0 = s_0$ and $y_{n-1}^0=s_0$.
	
\end{proof}

\section{Relation with the $\mu$-synthesis problem}
The $\mu$-synthesis problem plays an important role in robust control theory of control engineering.
Here $\mu$ is used to denote the structured singular value of a matrix relative
to a space of linear transformations. For a given linear subspace $E$ of $\C^{n \times m}$ the {\em structured singular value} of an $m \times n$ matrix $B$, denoted by $\mu_E(B)$, is defined as
\[
\mu_E(B) = \dfrac{1}{\left( \inf \{ \n X \n : X \in E, I-BX  \mbox{ is singular} \} \right)}.
\]
Here $\n X\n$ denotes the operator norm of the matrix $X$. In the event of $I-BX$ is non-singular for all $X \in E$, we define $\mu_E(B)=0$. The underlying linear subspace $E$ is considered as a \textquotedblleft structure \textquotedblright.

Given a linear subspace $E\subset \C^{n \times m}$, distinct points $\lm_1, \dots, \lm_k$ in $\D$ and same number of $m \times n$ target matrices $B_1, \dots, B_k$, the $\mu$-synthesis problem aims to find an analytic matrix-valued function $f : \D \longrightarrow \C^{m \times n}$ such that $f(\lm_j)=B_j $ for $1 \leq j \leq k$ and
$\mu_E(f(\lambda))<1$, for all $\lambda \in \mathbb D$. 

On the other hand, for given distinct points $\lm_1, \dots, \lm_k$ in $\D$ and $n \times n$ matrices $W_1,\dots, W_k$, the spectral Nevanlinna-Pick interpolation problem is to find necessary and sufficient conditions for the existence of an analytic matrix-valued function $F : \D \longrightarrow \C^{n \times n}$ such that $F(\lm_j)=W_j $ for $1 \leq j \leq k$ and
$r(F(\lambda))<1$, for all $\lambda \in \mathbb D$ ($r(F(\lambda))$ denotes the spectral radius of $F(\lambda)$). In \cite{costara1}, it was shown that such interpolation problem into the spectral unit ball of $\C^{n \times n}$ is equivalent to the interpolation problem into the symmetrized polydisc $\gn$.

Note that $\mu_E (B)= \|B\|$, when $E=\C^{n \times m}$. For the case of $m=n$ and $E = \ls \lambda I_{n} : \lambda \in \C \rs$, $I_n$ denotes the $n \times n$ identity matrix, $\mu_E(B)$ is equal to the spectral radius $r(B)$ and consequently in this case the $\mu$-synthesis problem reduced to the spectral Nevenlina-Pick interpolation  problem. For more details about the $\mu$-synthesis problem and structured singular value a reader is referred to \cite{BFT90,doyle}.

In this section we describe the relation of the domain $\Gn$ with the $\mu$-synthesis problem. In Theorem \ref{mu-syn1}, we show that the points of $\Gn$ can also be characterized using the structured singular value. In Remark \ref{mu interpolation} we show the connection between the $\mu$-synthesis problem with the interpolation problem in $\textrm{Hol}(\D, \Gn)$.  

Consider the linear subspace $E = \ls \begin{pmatrix}
z & 0\\
0 & w
\end{pmatrix} : z,w \in \C \rs$. 
For a $2 \times 2$ matrix $B=[b_{ij}]$ and $X = \begin{pmatrix}
z & 0\\
0 & w
\end{pmatrix}$, we have $\n X \n = \max\{ |z|, |w| \}$ and 
\[
\det(I - BX) = 1- z b_{11} - w b_{22} - zw \det B .
\]
Then for $r>0$, 
\begin{align*}
&\mu_E(B) \leq \frac{1}{r} \\
\Leftrightarrow & \max\{ |z|, |w| \} \geq r \; \text{ whenever }\; 1- b_{11}z - b_{22}w - \det B zw =0 \\
\Leftrightarrow &  1- z b_{11} - w b_{22} - zw \det B  \neq 0 \; \text{ for any  } z,w \in r\D.
\end{align*}

\begin{thm} \label{mu-syn1}
$y \in \Gn$ if and only if there exist $\lt \frac{n}{2} \rt$ number of $2 \times 2$ matrices $B_1, \dots, B_{[\frac{n}{2}]}$ such that $\det B_1=\dots=B_{[\frac{n}{2}]}$,  $y= \pi_n \lf B_1, \dots, B_{[\frac{n}{2}]} \rf $ and $\mu_{E}(B_j) <1$ for $j = 1, \dots, [\frac{n}{2}]$.
\end{thm}

\begin{proof}
	Suppose $y \in \Gn$. Then for all $j=1, \dots = [\frac{n}{2}]$, by Theorem $2.5$ of \cite{pal-roy 4}, we have
	\begin{align*}
	1 - \frac{y_j}{{n \choose j}} z - \frac{y_{n-j}}{{n \choose j}} w +  qzw \neq 0 \text{ whenever } z,w \in r\D \text{ for some } r>1.
	\end{align*}
	For each $1 \leq j \leq [\frac{n}{2}]$, consider the matrix 
	$$
	\begin{bmatrix}
	 y_j/{n \choose j} & w_j \\
	 \\
	w_j &  y_{n-j}/{n \choose j}
	\end{bmatrix} = B_j \;\text{ (say)},
	$$
	where ${w_j}^2 = \frac{y_j y_{n-j} - \nj^2 q}{{n \choose j}^2}$. Then $\mu_E(B_j) \leq \dfrac{1}{r}<1$. Clearly $\pi_n \lf B_1, \dots, B_{[\frac{n}{2}]} \rf = y$.
	
	Conversely, suppose there exist matrices $B_1, \dots, B_{[\frac{n}{2}]}$ such that $\det B_1=\dots=B_{[\frac{n}{2}]}$,  $y= \pi_n \lf B_1, \dots, B_{[\frac{n}{2}]} \rf $ and $\mu_{E}(B_j) <1$ for $j = 1, \dots, [\frac{n}{2}]$. Then for each such $j$ there is some $r_j>1$ so that $\mu_E(B_j) \leq \dfrac{1}{r_j}$. Therefore
	\begin{align*}
	1 - z [B_j]_{11}  - w [B_j]_{22} - zw \det B_j \neq 0 \; \text{ for all  } z,w \in \overline\D.
	\end{align*}
	Note that $ \pi_n \lf B_1, \dots, B_{[\frac{n}{2}]} \rf =y$ implies $[B_j]_{11}=\dfrac{y_j}{\nj}$, $[B_j]_{22}=\dfrac{y_{n-j}}{\nj}$ and $\det B_j =q$.
	Hence we have
	\[
	{n \choose j} - y_j z - y_{n-j}w + { n \choose j} qzw \neq 0,\; \text{ for all } z,w \in \overline\D.
	\]
	Therefore $y \in \Gn$.
\end{proof}

\begin{rem}
\begin{enumerate}
\item With the similar line of argument we can show that $y \in \Gamn$ if and only if there exist $\lt \frac{n}{2} \rt$ numbers of $2 \times 2$ matrices $B_1, \dots, B_{[\frac{n}{2}]}$ such that $\det B_1=\dots=B_{[\frac{n}{2}]}$,  $y= \pi_n \lf B_1, \dots, B_{[\frac{n}{2}]} \rf $ and $\mu_{E}(B_j) \leq 1$ for $j = 1, \dots, [\frac{n}{2}]$.\\
\item If a point belongs to the special set $\mathcal J_n \subset \Gn$, then instead of $[\frac{n}{2}]$ number of matrices the existence of only one matrix is sufficient. We can prove, with similar method as in the proof of last theorem, that $y \in \mathcal J_n$ if and only if there exist a $2 \times 2$ matrix $B$ such that $ \widehat\pi_n (B) = y $ and $\mu_{E}(B) <1$. Indeed, for $y \in \mathcal J_n$ the matrix 
$
B=\begin{bmatrix}
y_1/n &  w\\
\\
w &  y_{n-1}/n
\end{bmatrix},
$
where ${w}^2 = \dfrac{y_1 y_{n-1} - n^2 q}{n^2}$, satisfies required conditions; on the other hand for any $2 \times 2$ matrix $B$ such that $y= \widehat\pi_n \lf B \rf $ and $\mu_{E}(B) <1$, we have $n - y_1 z - y_{n-1}w + n qzw \neq 0$, for all  $z,w \in \overline\D$ which implies $y \in \mathcal J_n$.
\end{enumerate}
\end{rem}

Let $\Omega_E$ be the unit $\mu_E$-ball, that is, $\Omega_E = \ls B \in \C^{2 \times 2} : \mu_E(B) < 1 \rs $.
Every analytic map $\phi : \D \to \Gn$ gives rise to $[\frac{n}{2}]$ numbers of analytic functions $F_1, \dots, F_{[\frac{n}{2}]}$ from $\D$ to $\Omega_E$ such that $\pi_n \lf F_1, \dots, F_{[\frac{n}{2}]} \rf = \phi$. Indeed, write $\phi = (\phi_1, \dots, \phi_n)$ and consider the matrix valued function 
\[
F_j = \begin{bmatrix}
\phi_j & \phi_j \phi_{n-j} - \phi_n \\
1 & \phi_{n-j}
\end{bmatrix}.
\]
It can be seen, with arguments similar to the proof of Theorem $\ref{interpolating 1}$, that $\mu_{E}(F_j(\lm)) <1 $ for all $\lm \in \D$. Clearly each $F_j$ is analytic on $\D$ and $\pi_n \lf F_1, \dots, F_{[\frac{n}{2}]} \rf = \phi$. In the following remark we find a necessary condition for the $\mu$-synthesis problem where the underlying linear subspace is the space of $2 \times 2$ diagonal matrices.
\begin{rem}\label{mu interpolation}
	Let $\lm_1, \dots, \lm_m$ be $m$ distinct points in $\D$ and $B_1, \dots, B_m \in \Omega_E$. If there exists an analytic function $F : \D \longrightarrow \Omega_E$ such that $F(\lm_j)= B_j$ for $1 \leq j \leq m$, then there exists an analytic function $\phi : \D \longrightarrow \mathcal J_n$ such that $\phi (\lm_j)= \widehat{\pi_n}(B_j)$ for $1 \leq j \leq m$. In fact, we may take $\phi = \widehat{\pi_n} \circ F$. The problem of finding such $F$ is called the \textit{structured Nevanlinna-Pick problem} and clearly it has a connection with the finite interpolation problem in $\textrm{Hol}(\D, \Gn)$.
\end{rem}

\section{Invariant distances for a subset of $\widetilde{\mathbb G}_n$}

For a domain $\Omega \subset \C^n$ and two points $z,w \in \Omega$, the \textit{Carath\'{e}odory pseudo-distance} between $z,w$ is
$$ \mathcal C_{\Omega}(z,w) := \sup\big\{ \rho(f(z), f(w)) : f\in \mathcal O(\Omega,\D)  \big\}$$ 
and the \textit{Lempert function} for $\Omega$ is defined as
$$ \mathcal L_{\Omega}(z,w) := \inf \big\{ \rho(\al,\be) : \text{ there is } f\in \mathcal O(\D, \Omega),\; f(\al) = z, f(\be)=w   \big\},$$
where $\rho$ is the hyperbolic distance. It can be seen that for $z,w\in \D$, $ \rho(z,w) = \tanh^{-1}(d(z,w))$ where $d$ is the M\"{o}bius distance. The famous Lempert theorem states that $\mathcal C_{\mathcal D} = \mathcal L_{\mathcal D}$ for a domain $\mathcal D \subset \C^n$ if $\mathcal D$ is convex (see \cite{Le86}). 
In 2004, Agler and Young proved (see \cite{AY04}) that these two distances coincide for the non-convex domain $\mathbb G_2$. Later in 2007, Nikolov, Pflug and Zwonek proved in \cite{zwonek5} that they do not agree for $\mathbb G_n$ if $n\geq 3$. 
The tetrablock, a domain in $\C^3$, was shown to be the second example of a non-convex domain in which those two distances coincide (see \cite{awy, EKZ1}). Since $\mathbb E$ is linearly isomorphic to $\widetilde{\mathbb G}_3$, evidently Lempert's theorem holds for $\widetilde{\mathbb G}_3$. 
We still do not know whether Lempert theorem holds for $\Gn$, $n > 3$ or not, but we shall prove that $\mathcal C_{\Gn} = \mathcal L_{\Gn}$ for the points $z,w$ in $\mathcal J_n \subset \Gn$, where at least one of $z,w$ is equal to $0$. We apply the same techniques as in \cite{awy} to establish this.

\begin{thm}\label{CKL 1}
	Let $y = (y_1,\dots, y_{n-1},q) \in \mathcal J_n \subset \Gn$, then
	\begin{align*}
	\mathcal C_{\Gn}(0,y) = \mathcal L_{\Gn}(0,y) 
	= \max_{1\leq j \leq n-1} \ls \tanh^{-1} \dfrac{{n \choose j} \left|y_j - \bar y_{n-j} q \right| + \left|y_j y_{n-j} - {n \choose j}^2 q \right|}{{n \choose j}^2 - |y_{n-j}|^2 }\rs.
	\end{align*}
\end{thm}
\begin{proof}
	It is known that $ \mathcal C_{\Omega} \leq \mathcal L_{\Omega}$ for any domain $\Omega \subset \C^n$. We only show $\mathcal C_{\Gn}(0,y) \geq \mathcal L_{\Gn}(0,y)$ whenever $y \in \mathcal J_n$. One can write
	\begin{align*} 
	\mathcal L_{\Gn}(0,y) := \inf \big\{ \tanh^{-1}|\lambda| : &\text{ there exists } f\in \mathcal O(\D, \Gn) \\ 
	& \q \text{ such that } f(0) = (0,\dots,0), f(\lambda)=y   \big\}.
	\end{align*}
	For any $f\in \mathcal O(\D, \Gn)$ satisfying $f(0) = (0,\dots,0)$ and $f(\lambda)=(y_1,\dots, y_{n-1},q)$, using the Schwarz lemma for $\Gn$ in \cite{pal-roy 5}, we have
	$
	|\lm| \geq \max\limits_{1\leq j \leq n-1} \ls \n \Phi_j(.,y)\n_{H^{\infty}} \rs.
	$
	Thus
	\[
	\mathcal L_{\Gn}(0,y)  \geq  \max\limits_{1\leq j \leq n-1} \ls \tanh^{-1} \n \Phi_j(.,y)\n_{H^{\infty}} \rs.
	\]
	Now consider 
	$ \lambda_0 = \max\limits_{1\leq j \leq n-1} \ls \n \Phi_j(.,y)\n_{H^{\infty}} \rs.$
	Since $y \in \mathcal J_n$, again by the Schwarz lemma for $\Gn$ in \cite{pal-roy 5}, there exists $f\in \mathcal O(\D, \Gn)$ satisfying $f(0) = (0,\dots,0)$ and $f(\lambda_0)=y$. Therefore 
	\begin{align*}
	\nonumber \mathcal L_{\Gn}(0,y) & =  \max_{1\leq j \leq n-1} \ls \tanh^{-1} \n \Phi_j(.,y)\n_{H^{\infty}} \rs \\
	& = \max\limits_{1\leq j \leq n-1} \ls \tanh^{-1} \dfrac{{n \choose j} \left|y_j - \bar y_{n-j} q \right| + \left|y_j y_{n-j} - {n \choose j}^2 q \right|}{{n \choose j}^2 - |y_{n-j}|^2 }\rs. 
	\end{align*}
	
	The Carath\'{e}odory pseudo-distance can be expressed as
	$$ \mathcal C_{\Gn}(0,y) = \sup \big\{\tanh^{-1}|f(y)| :  f \in \mathcal O(\Gn,\D), f((0,\dots,0))=0 \big\}.$$
	For any $\om \in \T$ and any $j \in \{ 1, \dots, n-1 \}$, the function $\Phi_j(\om,.)$ is analytic from $\Gn$ to $\D$ and satisfies $\Phi_j(\om,(0,\dots,0))=0$. Thus
	$$\mathcal C_{\Gn}(0,y) \geq \tanh^{-1}|\Phi_j(\om,y)|\q \text{for any } \om \in \T.$$
	Hence $\mathcal C_{\Gn}(0,y) \geq  \tanh^{-1}\n \Phi_j(.,y)\n_{H^{\infty}}$ for any $1\leq j \leq n-1$.
	Therefore 
	$$\mathcal C_{\Gn}(0,y) \geq \max_{1\leq j \leq n-1} \ls \tanh^{-1}\n \Phi_j(.,y)\n_{H^{\infty}} \rs =  \mathcal L_{\G}(0,y).$$
	
\end{proof}

\begin{rem}
	By the last part of the proof of the last theorem it is clear that, for any point $y \in \Gn$, we have 
	$$\mathcal C_{\Gn}(0,y) \geq \max_{1\leq j \leq n-1} \ls \tanh^{-1}\n \Phi_j(.,y)\n_{H^{\infty}} \rs.$$
\end{rem}

\section{appendix}

Calculation for the implication $\ref{appendix 1} \Rightarrow \ref{appendix 1.1}$: \\
\begingroup
\allowdisplaybreaks
\begin{align*}
& (1-|\lm_0|^2 Z_{\sg,j}^*Z_{\sg,j}) =\begin{bmatrix}
1-\df{|y_j^0|^2}{\nj^2} - \df{|\lm_0|^2|w|^2}{\sg^2}  & -\df{\lm_0\bar{y}_j^0\sg w}{\nj}- \df{|\lm_0|^2y_{n-j}^0\bar{w}}{\nj\sg}   \\
\\
-\df{\bar{\lm}_0y_j^0\sg \bar{w}}{\nj}- \df{|\lm_0|^2\bar{y}_{n-j}^0w}{\nj\sg} & 1 - |\lm_0|^2\sg^2|w|^2 -\df{|\lm_0|^2|y_{n-j}^0|^2}{\nj^2}
\end{bmatrix},\\
&(1-Z_{\sg,j}^*Z_{\sg,j})^{-1} = \df{1}{\det(1-Z_{\sg,j}^*Z_{\sg,j})}\begin{bmatrix}
1 - \sg^2|w|^2 -\df{|y_{n-j}^0|^2}{\nj^2}   & \df{\bar{y}_j^0\sg w}{\nj \bar{\lm}_0}+ \df{y_{n-j}^0\bar{w}}{\nj \sg}   \\
\\
\df{y_j^0\sg \bar{w}}{\nj \lm_0}+ \df{\bar{y}_{n-j}^0w}{\nj \sg} &  1-\df{|y_j^0|^2}{\nj^2|\lm_0|^2}- \df{|w|^2}{\sg^2}
\end{bmatrix}.
\end{align*}
\endgroup
Then,
\begin{align}\label{app-4}
\nonumber & [\mathcal{K}_{Z_{\sg,j}}(|\lm_0|)\det (1-Z_{\sg,j}^*Z_{\sg,j})]_{11} \\
\nonumber & = \det (1-Z_{\sg,j}^*Z_{\sg,j})[(1-|\lm_0|^2 Z_{\sg}^*Z_{\sg})(1-Z_{\sg}^*Z_{\sg})^{-1}]_{11} \\
&= 1 - \df{|y_j^0|^2}{\nj^2} - \df{|y_{n-j}^0|^2}{\nj^2} + |q^0|^2 -
\df{|y_j^0y_{n-j}^0-\nj^2q^0|}{\nj^2}\Big(\frac{|\lm_0|}{\sg^2} +\frac{\sg^2}{|\lm_0|}\Big) .
\end{align}
Note that, $ \det (1-Z_{\sg,j}Z_{\sg,j}^*) = \det(1-Z_{\sg,j}^*Z_{\sg,j})$ and
$$(1-Z_{\sg,j}Z_{\sg,j}^*)^{-1} = \df{1}{\det(1-Z_{\sg,j}^*Z_{\sg,j})} \begin{bmatrix}
1 - \df{|w|^2}{\sg^2} -\df{|y_{n-j}^0|^2}{\nj^2}  &   \df{y_j^0 \bar{w}}{\nj \sg\lm_0}+ \df{\bar{y}_{n-j}^0\sg w}{\nj}\\
\\
\df{\bar{y}_j^0 w}{\nj \sg\bar{\lm}_0}+ \df{y_{n-j}^0\sg\bar{w}}{\nj} &  1-\df{|y_j^0|^2}{\nj^2|\lm_0|^2} -\sg^2|w|^2
\end{bmatrix}.$$
It follows that
\begingroup
\allowdisplaybreaks
\begin{align*}
&[\mathcal{K}_{Z_{\sg,j}}(|\lm_0|)\det (1-Z_{\sg,j}^*Z_{\sg,j})]_{12}\\
&= \det(1-Z_{\sg,j}^*Z_{\sg,j}) (1-|\lm_0|^2) [(1-Z_{\sg,j}Z_{\sg,j}^*)^{-1}Z_{\sg,j}]_{21}
= (1-|\lm_0|^2)\Big( \df{w}{\sg} + \df{q^0}{\lm_0}\sg\bar{w} \Big).
\end{align*}
\endgroup
Similarly we have
$$[\mathcal{K}_{Z_{\sg,j}}(|\lm_0|)\det (1-Z_{\sg,j}^*Z_{\sg,j})]_{21} = (1-|\lm_0|^2)\Big( \df{\bar{w}}{\sg} + \df{\bar{q}^0}{\bar{\lm}_0}\sg w \Big).$$
Clearly
\[
(Z_{\sg,j}Z_{\sg,j}^* - |\lm_0|^2) 
= \begin{bmatrix}
\df{|y_j^0|^2}{\nj^2|\lm_0|^2} +\sg^2|w|^2 -|\lm_0|^2  &   \df{y_j^0 \bar{w}}{\nj \sg\lm_0} + \df{\bar{y}_{n-j}^0\sg w}{\nj}\\
\\
\df{\bar{y}_j^0 w}{\nj \sg\bar{\lm}_0} +\df{y_{n-j}^0\sg\bar{w}}{\nj}  & \df{|w|^2}{\sg^2} +\df{|y_{n-j}^0|^2}{\nj^2} -|\lm_0|^2
\end{bmatrix}
\]
Therefore
\begin{align*}
&[\mathcal{K}_{Z_{\sg,j}}(|\lm_0|)\det (1-Z_{\sg,j}^*Z_{\sg,j})]_{22}\\
&= \det (1-Z_{\sg,j}^*Z_{\sg,j})[(Z_{\sg}Z_{\sg}^* - |\lm_0|^2)(1-Z_{\sg}Z_{\sg}^*)^{-1}]_{22} \\
&= -|\lm_0|^2 - \df{|y_j^0|^2}{\nj^2} - \df{|y_{n-j}^0|^2}{\nj^2} +\df{|q^0|^2}{|\lm_0|^2} - \df{|y_j^0y_{n-j}^0-\nj^2q^0|}{\nj^2}\Big(|\lm_0|\sg^2 +\frac{1}{\sg^2|\lm_0|}\Big).
\end{align*}

\noindent\textbf{Acknowledgments}\\
The author is supported by the Institute postdoctoral fellowship of TIFR-CAM. The author would like to thank Pof. Sourav Pal for so many helpful discussion and suggestions.

The author would also like to thank the referee for making several valuable suggestions and comments which help in improving the exposition of the paper.

\end{document}